\documentclass[a4paper,11pt]{article}

\usepackage{makeidx}
\usepackage{amscd}
\usepackage{amsmath}
\usepackage{amssymb}
\usepackage{amsthm}

\usepackage[all]{xy}

\theoremstyle{plain}
\newtheorem{thm}{Theorem}[section]
\newtheorem{cor}[thm]{Corollary}
\newtheorem{lem}[thm]{Lemma}
\newtheorem{cla}[thm]{Claim}

\newtheorem{prop}[thm]{Proposition}
\newtheorem{conj}[thm]{Conjecture}
\newtheorem{prob}[thm]{Problem}

\theoremstyle{definition}
\newtheorem{defn}[thm]{Definition}
\newtheorem{rem}[thm]{Remark}
\newtheorem{exa}[thm]{Example}

\newtheorem{defn-prop}[thm]{Definition-Proposition}

\newcommand{\PP}{\mathbb P}
\newcommand{\Z}{\mathbb Z}
\newcommand{\Q}{\mathbb Q}

\newcommand{\C}{\mathbb C}

\newcommand{\Image}{\mathop{\mathrm{Im}}\nolimits}

\newcommand{\id}{\ensuremath{\mathop{\mathrm{id}}}}

\newcommand{\rk}{\ensuremath{\operatorname{rank}}}

 \newcommand{\Cone}{\operatorname{Cone}} 
 
\newcommand{\Pic}{\mathop{\mathrm{Pic}}\nolimits}
\newcommand{\Ext}{\mathop{\mathrm{Ext}}\nolimits}
\newcommand{\Hom}{\mathop{\mathrm{Hom}}\nolimits}

\newcommand{\RHom}{\mathop{\mb R\mathrm{Hom}}\nolimits}

\newcommand{\Lotimes}{\stackrel{\mb L}{\otimes}}

\newcommand{\SL}{\operatorname{SL}}
\newcommand{\Coh}{\operatorname{Coh}}

\newcommand{\Auteq}{\operatorname{Auteq}}

\newcommand{\Aut}{\operatorname{Aut}} 
\newcommand{\ch}{\operatorname{ch}} 
\newcommand{\Kequiv}{\operatorname{K-equiv}}

\newcommand{\mc}{\mathcal}
\newcommand{\mb}{\mathbb}

\newcommand{\Supp}{\ensuremath{\operatorname{Supp}}}

\newcommand{\Br}{\ensuremath{\operatorname{Br}}}
\newcommand{\Comp}{\ensuremath{\operatorname{Comp}}}
\newcommand{\Span}[1]{\left<#1\right>}

\title{A trichotomy for the autoequivalence groups on  smooth projective surfaces}
\author{Hokuto Uehara}
\date{}
\begin{document}
\maketitle
\begin{abstract}
We study autoequivalence groups of the derived categories on smooth projective surfaces, and 
show a trichotomy of types according to the maximal dimension of Fourier--Mukai kernels for autoequivalences. This number is $2$, $3$ or $4$, and 
we also pose a conjecture on the description of autoequivalence groups if it is $2$, 
and prove it in some special cases.
\end{abstract}


\section{Introduction}
The study of derived categories $D(X)=D^b(\Coh(X))$ of coherent sheaves on a smooth projective varieties $X$
has become an important topic in algebraic geometry 
over the last decades. It is an interesting and basic problem to describe the group $\Auteq D(X)$
 of autoequivalences of $D(X)$. In this article, we consider the autoequivalence group of smooth projective surfaces. 

Let us introduce an integer $N_X$, which plays the key role of a trichotomy of types of the 
autoequivalence group on smooth projective surfaces.
First recall that 
an Orlov's deep result states that every autoequivalence on a smooth projective variety $X$ is given by a 
 \emph{Fourier--Mukai transform} $\Phi^\mathcal{P}$ with unique kernel $\mc{P}\in D(X\times X)$ (see \S \ref{sec:FM}). 
Let us define
$$
\Comp (\Phi^\mc{P})
$$
to be the set of irreducible components $W_0$ of $\Supp(\mc{P})$ which dominates $X$ by the first projection $p_1\colon X\times X\to X$,
which turns out to be non-empty by \ref{lem:fundamental_FM} (i)).
Define
\begin{align*}
N_X:&=\max \{\dim (W_0) \mid W_0\in \Comp (\Phi^\mc{P}) \text{ for some }\Phi^{\mc{P}}\in \Auteq D(X) \}\\
     &\in\{\dim (X),\dim (X)+1,\ldots , 2\dim (X)\}
\end{align*}
and call it the \emph{Fourier--Mukai support dimension} of $X$.

For smooth projective surfaces $S$,  (conjectural) descriptions of 
the group $\Auteq D(S)$ and the geometry of surfaces
are quite different, depending on the value $N_S$. 
The following is the first main result in this article.  


\begin{thm}[=Theorem \ref{thm:trichotomy}]\label{thm:trichotomy0}
We have the following.
\begin{enumerate}
\item
$N_S=4$ if and only if $K_S\equiv 0$. 
\item
$N_S=3$ if and only if $S$ has a minimal elliptic fibration and $K_S\not \equiv 0$. 
\item
$N_S=2$ if and only if $S$ has no minimal elliptic fibration and $K_S\not \equiv 0$. 
\end{enumerate}
\end{thm}

In the case $N_S=4$, Theorem \ref{thm:trichotomy0} implies that 
$S$ is one of K3, abelian, bielliptic or Enriques surfaces.
Bayer and Bridgeland describe the auto\-equivalence group of  K3 surfaces with the Picard number $1$ \cite{BB17}.
Orlov finds a description of  the autoequivalence group of  abelian varieties (not necessarily surfaces) \cite{Or02}.
Recently, Potter finds a description of the autoequivalence group of  bielliptic surfaces \cite{Po17}. 

Let us consider the case $N_S=3$. In this case, Theorem \ref{thm:trichotomy0} implies that
 $S$ has a minimal elliptic fibration $\pi\colon S\to C$ and $K_S\not\equiv 0$. 
Suppose furthermore that each reducible fiber of $\pi$ is non-multiple, and forms a cycle of $(-2)$-curves,
 i.e. it is of type $\mathrm{I}_n$ for some $n>1$. 
Then, the autoequivalence group $\Auteq D(S)$ is described in \cite{Ue16}.
See also Conjecture \ref{conj:A}.

Finally, let us consider the case $N_S=2$.  
Let us set $Z$ the union of all $(-2)$-curves on $S$, and define
$$
\Br_Z(S)=\Span{T_{\alpha}\mid \alpha \in D_Z(S) \text{ spherical object}}(\subset \Auteq D(S)).
$$
Here, a functor $T_{\alpha}$ is a special kind of an autoequivalence, called a \emph{twist functor} (see \S \ref{subsec:twist_functor}).
Then, we pose the following conjecture:


\begin{conj}[cf. Conjecture \ref{conj:generation}]\label{conj:B}
If $N_S=2$, then we have
$$
\Auteq D(S)=\Span{\Br_Z(S), \Pic (S)} \rtimes \Aut (S) \times \Z[1].
$$
\end{conj}

The classical Bondal--Orlov Theorem states that 
if $\pm K_X$ is ample for a smooth projective variety $X$,
we have 
$$
\Auteq D(X)=\Pic (S) \rtimes \Aut (S) \times \Z[1].
$$
Because there are no $(-2)$-curves on a smooth projective surface $S$ with ample $\pm K_S$ 
we can regard Conjecture \ref{conj:B} as a variant  of their result.
The following is the second main result of this article.

\begin{thm}[cf. Theorem \ref{thm:generation}]\label{thm:generation0}
Let $S$ be a smooth projective surface with $N_S=2$.
Then Conjecture \ref{conj:B} holds true,
 if 
$Z$ is a disjoint union of configurations of $(-2)$-curves of type $A$.
\end{thm}

Theorem \ref{thm:generation0} is a generalization of \cite[Theorem 1.5]{IU05} and \cite[Theorem 1]{BP14}.
We show Theorem \ref{thm:generation}, which is slightly stronger than Theorem \ref{thm:generation0}.

\paragraph{Notation and conventions.}\label{para:notation_convention}
We follow the notation and terminology of \cite{Ha77} unless otherwise stated.
All varieties will be defined over the complex number field $\C$ in this article.
A \emph{point} on a variety will always mean a closed point.

By a \emph{minimal elliptic surface}, we will
always mean a smooth projective surface $S$ together with a smooth projective
curve $C$ and a relatively minimal morphism $\pi\colon S\to C$ whose general fiber is an elliptic curve. Here  a \emph{relatively
minimal morphism} means a morphism whose fibers contains no $(-1)$-curves. Such a morphism $\pi$ is 
called an \emph{minimal elliptic fibration}.

We  denote by $D(X)$ the bounded derived category of coherent sheaves on an algebraic variety $X$
For any subset $Z(\subset X)$, we denote the full triangulated subcategory of $D(X)$ consisting of objects supported on $Z$ by $D_Z(X)$.
Here, the support of an object $\alpha\in D(X)$ is, by definition, the union of the set-theoretic supports of its cohomology sheaves $\mc{H}^i (\alpha)$.  Note that the support is always closed subset because $\alpha$ is a bounded complex of coherent sheaves.
We denote the dimension of the support of $\alpha$ by $\dim (\alpha)$. 

An object $\alpha$ in $D(X)$ is said to be \emph{rigid} if $\Hom^1_{D(X)}(\alpha,\alpha)= 0$.

Given a closed embedding of schemes  $i\colon Z\hookrightarrow X$,
we denote the derived pullback 
$\mb Li^*\alpha$ simply by $\alpha|_Z$. 

For algebraic varieties $X,Y$,
we denote the diagonal in $X\times X$ by $\Delta_{X}$, 
and denote the projections by $p_X\colon X\times Y\to X$ and $p_Y\colon X\times Y\to Y$,
or $p_1\colon X\times Y\to X$ and $p_2\colon X\times Y\to Y$.

For an abelian variety $X$, we denote the dual variety $\Pic ^0X$ by $\widehat{X}$.

$\Auteq \mathcal{T}$ denotes the group of isomorphism classes of $\C$-linear exact autoequivalences of a $\C$-linear triangulated category $\mathcal{T}$.

For a Cartier divisor $D$ on a normal projective variety $X$,
we define a graded $\C$-algebra by  
$$
R(X,D):=\bigoplus _{m\ge 0}H^0(X,\mc{O}_X(mD)).
$$
Recall that the  \emph{Iitaka dimension} $\kappa (X,D)=\kappa (D)(\in \{-\infty, 0, 1,\ldots, \dim (X)\})$ of $D$ is
$$
\kappa (D):=
\begin{cases}
\text{ the transcendence degree of } R(X,D) -1 &\quad \text{ if } R(X,D)\ne \C
\\
-\infty &\quad \text{ otherwise.}
\end{cases}
$$
We call $\kappa(K_X)$ the \emph{Kodaira  dimension} of $X$, 
and simply denote it by $\kappa (X)$.
Assume furthermore that $D$ is a nef divisor. 
Then,  recall the \emph{numerical Iitaka dimension} $\nu (X,D)=\nu (D)(\in \{0, 1,\ldots, \dim (X)\})$ of $D$  by
$$
\nu (D):=\max \{k \in \Z \mid D^k\cdot H^{\dim (X)-k}\ne 0\},
$$
where $H$ is an ample divisor on $X$. In general, it is known that the inequality
$$
\nu (D)\ge \kappa (D)
$$
holds.

Let $X$ be a minimal model, that is, $X$ is a normal projective variety with $\Q$-factorial terminal singularities and $K_X$ is nef.  
We call $\nu (K_X)$  the \emph{numerical Kodaira  dimension} of $X$, 
and simply denote it by $\nu (X)$.
The \emph{abundance conjecture} states that  if $X$ is a minimal model, then the equality 
$\kappa(X)=\nu (X)$ holds. It is known to be true for surfaces and $3$-folds.
See \cite{KMM} for these terminology and results.

\section{Fourier--Mukai transforms}\label{sec:preliminaries}
\subsection{Fourier--Mukai transforms}\label{sec:FM}
Let $X$ and $Y$ be smooth projective varieties.
For an
object $\mathcal{P}\in D(X\times Y)$, we define an exact functor $\Phi^{\mathcal{P}}$, 
called the \emph{integral functor} with \emph{kernel} $\mathcal{P}$, by
$$
\Phi^{\mathcal{P}}:= 
\mathbb{R}p_{Y*}(\mathcal{P}\Lotimes p^{*}_X(-))\colon D(X)\to D(Y).
$$
We also sometimes write $\Phi^{\mathcal{P}}$ as $\Phi^{\mathcal{P}}_{X\to Y}$ to emphasize that it is a functor from $D(X)$ to $D(Y)$.

By the result of Orlov (see \cite[Theorem 5.14]{Hu06}),
for a fully faithful functor $\Phi\colon D(X)\to D(Y)$,
there is an object $\mathcal{P}\in D(X\times Y)$, unique up to isomorphism, such that 
$
\Phi\cong \Phi^{\mathcal{P}}.
$
If an integral functor $\Phi^{\mc{P}}$ is an equivalence,
it is called a \emph{Fourier--Mukai transform}.

Note that every autoequivalence is given as an integral functor by the Orlov's result, and  hence let us consider \emph{standard autoequivalences}  as examples of Fourier--Mukai transforms.
The autoequivalence group $\Auteq D(X)$ always contains the group
$$
A(X):=\Pic (X) \rtimes \Aut (X)\times \Z[1],
$$
generated by standard autoequivalences, namely the functors of tensoring with line bundles, push forward along automorphisms, 
and the shift functor $[1]$.
Any standard equivalence $\Phi$ are given by the following form;
$$
\Phi=\varphi_*\circ ((-)\otimes \mc{L}) \circ [i]
$$
for an automorphism $\varphi$, an integer $i$ and a  line bundle $\mc{L}$.
Then, $\Phi$ is the Fourier--Mukai transform with the kernel 
\begin{equation}\label{eqn:standard}
(\mc{O}_{\Gamma_\varphi} \otimes p_1^*\mc{L})[i],
\end{equation}
whose support is $\Gamma_\varphi$, where $\Gamma_\varphi$ is the graph of $\varphi$.

Every Fourier--Mukai transform $\Phi^\mc{P}$ induces a \emph{cohomological Fourier--Mukai transform}
$$
\Phi^{\mathcal{P},H}\colon H^*(X,\Q)\to H^*(Y,\Q),
$$
which is an isomorphism of the total cohomologies, 
and the commutativity
$$
\Phi^{\mathcal{P},H}\circ v(-)= v(-)\circ\Phi^{\mathcal{P}}
$$
holds (see \cite[\S 5.2]{Hu06}).
Here we put $v(-):=\ch(-)\sqrt{\text{td}(X)}$.

If there exists a Fourier--Mukai transform between $D(X)$ and $D(Y)$, 
then we call $X$ a \emph{Fourier--Mukai  partner} of $Y$.

\subsection{Calabi--Yau objects and spherical objects}\label{subsec:twist_functor}
Let $X$ be a smooth projective variety. An object $\alpha\in D(X)$ is called a \emph{Calabi--Yau object} if it satisfies  
\begin{equation}\label{eqn:alpha_omega}
\alpha\otimes \omega_X\cong \alpha.
\end{equation}
For example, a $0$-dimensional sheaf on a smooth projective variety and a line bundle $\mc{L}$ 
on a $(-2)$-curve $C$ on a smooth projective surface are  Calabi--Yau objects.

Take a Calabi--Yau object $\alpha$, an autoequivalence $\Phi\in \Auteq D(X)$ and a closed subscheme $D$ of $\Supp(\mc{P})$. Then 
all cohomology sheaves 
$
\mc{H}^i( \alpha)
$
and $\alpha|_D$ are Calabi--Yau objects.
It is known that the Serre functor $(-)\otimes \omega _{X}[\dim (X)]$ commutes with the equivalence $\Phi$
 (cf.~\cite[Lemma~1.30]{Hu06}),
and thus, $\Phi(\alpha)$ is also a Calabi--Yau object.

Next, let us consider a sheaf $\mc{F}\in \Coh (X)$ which is a Calabi--Yau object (we call it a \emph{Calabi--Yau sheaf}).
Then, we have 
$$
\ch (\mc{F})=\ch (\mc{F})\cdot \ch (\omega_X)=\ch (\mc{F})\cdot (1+c_1(\omega_X) +\frac{1}{2}c_1(\omega)^2+\cdots)
$$ 
and hence,
\begin{align}\label{ali:CY-property}
0=\ch (\mc{F})\cdot  (c_1(\omega_X) +\frac{1}{2}c_1(\omega)^2+\cdots).
\end{align}

For a Calabi--Yau  object $\alpha\in D(X)$ 
and an irreducible curve  $C$ contained in $\Supp (\alpha)$,
every cohomology sheaf $\mc{H}^i(\alpha|_C)$ is a Calabi--Yau sheaf. Hence, 
equality \eqref{ali:CY-property} yields 
\begin{equation*}
K_X\cdot C=0.
\end{equation*}

If there exists a Calabi--Yau object $\alpha$ in $D(X)$ with $\Supp (\alpha)=X$,
we can find $i\in \Z$ such that $\rk \mc{H}^i(\alpha)>0$. Since $\mc{H}^i(\alpha)$ is also a Calabi--Yau sheaf, equality \eqref{ali:CY-property} implies that $c_1(\omega _X)$ is torsion.

Next we introduce an important class of examples of autoequivalences. 
We say that an object $\alpha \in D(X)$ is \emph{spherical} if 
$\alpha$ is a Calabi--Yau object and it satisfies 
$$
\Hom^{k}_{D(X)}(\alpha,\alpha)\cong\begin{cases}  0 & (k\ne 0,\dim (X))\\
                                                \C & (k=0,\dim (X)). 
\end{cases}
$$
For example, a line bundle on a K3 surface $X$ and a line bundle $\mc{L}$ on a 
$(-2)$-curve on a smooth projective surface $X$ are  spherical objects in $D(X)$ (see \cite[\S 2.2]{Ue16}). 

Put $X=X_1=X_2$.
For a spherical object $\alpha\in D(X)$, we consider the mapping cone 
\begin{equation}\label{eqn:kernel_twist}
\mc{C}:=\Cone(p_1^*\alpha^\vee\Lotimes p_2^*\alpha \to \mc{O}_{\Delta_X})\in D(X_1\times X_2)
\end{equation}
of the natural evaluation $p_1^*\alpha^\vee\Lotimes p_2^*\alpha \to \mc{O}_{\Delta_X}$.
Then the integral functor 
$T_{\alpha}:=\Phi^{\mc{C}}_{X_1\to X_2}$ defines an autoequivalence of $D(X)$,
called the \emph{twist functor} along the spherical object $\alpha$ 
(cf.~\cite[Proposition 8.6]{Hu06}). By \eqref{eqn:kernel_twist},  
there is a exact triangle 
\begin{equation}\label{eqn:twist_triangle}
\RHom_{D(X)}(\alpha,\beta)\Lotimes_\mathbb{C} \alpha\to \beta \to T_{\alpha} (\beta)
\end{equation}
for $\beta\in D(X)$.

\subsection{Fourier--Mukai transforms on elliptic surfaces}\label{subsec:bridgeland}
Refer \cite{Br98} to the results in this subsection.
Let $\pi\colon S\to C$ be a minimal elliptic surface. 
For an object $E$ of $D(S)$, we define the  \emph{fiber degree} of $E$ as
\[d(E)=c_1(E)\cdot F, \]
where $F$ is a general fiber of $\pi$. 
Let us denote  by $\lambda_{S}$  
the highest common factor of the fiber degrees of objects of $D(S)$. 
Equivalently,
$\lambda_{S}$ is the smallest number $d$ such that there exists a 
holomorphic $d$-section of $\pi$. 
Consider an integer $b$ coprime to $\lambda_{S}$. 
There exists a smooth,
$2$-dimensional component $J_{S/C}(b)=J_S (b)$ of the moduli space of pure one-dimensional 
stable sheaves on $S$,
the general point of which represents a rank $1$, degree $b$ stable 
vector bundle supported on a smooth fiber of $\pi$. 
There is a natural morphism $J_S (b)\to C$, taking a point representing
 a sheaf supported on the
fiber $\pi ^{-1}(x)$ of $S$ to the point $x$. This morphism is a minimal 
elliptic fibration.
Obviously, $J_S(0)\cong J(S)$, the Jacobian surface associated to $S$, 
and $J_S(1)\cong S$. 

There exists a universal sheaf $\mathcal{U}$ on $J_S(b)\times_C S$
such that the resulting functor $\Phi ^{\mathcal{U}}_{J_S(b)\to S}$ is an equivalence.

Let us set $Z$ the union of all $(-2)$-curves on $S$. Define
$$
\Br_Z(S)=\Span{T_{\alpha}\mid \alpha \in D_Z(S) \text{ spherical object}}(\subset \Auteq D(S)),
$$
and denote the congruence subgroup of $\SL (2,\Z)$ by
$$
\Gamma_0(m):=\Bigl\{ \begin{pmatrix}
c& a\\
d& b   
\end{pmatrix}\in \SL (2,\Z)\bigm|d\in m\Z \Bigr\}
$$
for $m\in \Z$.  
Then, we pose the following conjecture.


\begin{conj}\label{conj:A}
Suppose that a smooth projective surface $S$ has a minimal elliptic fibration $\pi\colon S\to C$ 
and $K_S\not\equiv 0$. Then, we have a short exact sequence
\begin{align*}
1\to \Span{\Br_Z(S),\otimes \mathcal{O}_S(D)\mid D\cdot F=0, \text{ $F$ is a fiber}}\rtimes \Aut (S)\times  
\Z[2]
\to
\Auteq D(S)& \notag\\
\stackrel{\Theta}\to 
\Big\{ \begin{pmatrix}
c& a\\
d& b   
\end{pmatrix}\in \Gamma_0(\lambda_{S}) \bigm| J_S(b)\cong S \Big\}
&\to 1.
\end{align*}
Here 
$\Theta$ is induced by the action of 
$\Auteq D(S)$ on the 
even degree part $H^0(F,\Z)\oplus H^2(F,\Z)\cong \Z^2$ of the integral cohomology group of  a smooth fiber $F$.
\end{conj}

Suppose that each reducible fiber of $\pi$ is non-multiple, and forms a cycle of $(-2)$-curves,
 i.e. it is of type $\mathrm{I}_n$ for some $n>1$. 
Then Conjecture \ref{conj:A} is shown to be true in \cite{Ue16}. See also \cite{Ue17}.


\section{Support of the kernel of Fourier--Mukai transforms}
In this section, we consider the support of  the kernel of Fourier--Mukai transforms.
Many results and ideas are due to Kawamata \cite{Ka02}, but for easy reference,
we often refer Huybrecht's book \cite{Hu06}.

Let $X$ and $Y$ be smooth projective varieties, and
suppose that 
$$\Phi=\Phi^{\mathcal{P}}_{X\to Y}\colon D(X)\to D(Y)$$ 
is a Fourier--Mukai transform. 
In this case, we have $\dim (X)=\dim (Y)$ (cf. \cite[Corollary 5.21]{Hu06}), and
the quasi-inverse of $\Phi$ is given by $\Phi^{\mc{Q}}$, where
$\mc{Q}=\mc{P}^\vee\otimes p_X^*\omega_X$.
It is known  (cf. \cite[Lemma 3.32]{Hu06}) that 
\begin{equation}\label{eqn:PQ}
\Supp (\mc{P})=\Supp(\mc{Q}).
\end{equation}
Let us denote by $\Gamma$ the support of $\mc{P}$. 
For $x\in X$, $\Gamma_x$ denotes the fiber over $x$ by $p_X|_{\Gamma}$.
Notice that
\begin{equation}\label{eqn:LiP}
\mathcal{P}|_{x\times X}\cong\Phi(\mathcal{O}_x)
\quad\text{and}\quad 
\Supp(\mc{P}|_{x\times X})=\Gamma_x 
\end{equation}
(see \cite[Lemma 3.29]{Hu06}),
which implies that 
$\Gamma_x=\Supp (\Phi(\mc{O}_x))$ as sets.
Furthermore, we note the following lemma.


\begin{lem}\label{lem:fundamental_FM}
\begin{enumerate}
\item
There exists an irreducible component of  $\Gamma$ which dominates $X$ by $p_X$,
and a similar statement holds for $p_Y$.
\item
$\Supp(\Phi (\mc{O}_x))$ is connected for any $x\in X$.
\item
If  $\dim (\Phi (\mc{O}_x))= \dim (X)$ holds, 
then $K_X\equiv 0$ and $K_Y\equiv 0$.
\item
Let $W$ be an irreducible component of $\Gamma$, and $\nu\colon \tilde{W}\to W$ be the normalization.
Then $\nu^*p_X^*\omega_X^{\otimes m}\cong \nu^*p_Y^*\omega_Y^{\otimes m}$ for some $m>0$. 
\end{enumerate}
\end{lem}

\begin{proof}
(i) See \cite[Lemma 6.4]{Hu06} for the first statement.  Apply the first statement for the quasi-inverse 
$\Phi^{\mc{Q}}_{Y\to X}$ and use \eqref{eqn:PQ}. Then, we can show the second.

(ii) See \cite[Lemma 6.11]{Hu06}.

(iii) If the equality holds, $\Supp (\Phi(\mc{O}_x))=Y$. Since $\Phi(\mc{O}_x)$ is a Calabi--Yau object, $c_1(Y)$ is torsion as in \S \ref{subsec:twist_functor}. 
Then the result follows by Remark \ref{rem:nef-nef} (ii).

(iv) See \cite[Lemma 6.9]{Hu06}.
\end{proof}

Let us define
$$
\Comp (\Phi^\mc{P}_{X\to Y})
$$
the set of irreducible components $W_0$ of $\Gamma=\Supp(\mc{P})$ which dominates $X$ by $p_X$.
Note that  $\Comp (\Phi^\mc{P}_{X\to Y})\ne \emptyset$ by Lemma \ref{lem:fundamental_FM} (i).


\begin{lem}\label{lem:i-dimensional}
Take an  irreducible component $W$ of  $\Gamma$.
\begin{enumerate}
\item  We see $\dim (W)\le \dim (p_X(W))+\dim (p_Y(W))$. If furthermore $\dim (W)= \dim (p_X(W))+\dim (Y)$ holds,
then $K_X\equiv 0$.
\item 
If $\dim (W)=\dim (X)$ and $W\in \Comp (\Phi^\mc{P}_{X\to Y})$, 
then $W$ is the unique irreducible component dominating $X$ by $p_X$. Furthermore, it also dominates $Y$ by $p_Y$.
\item
If $\dim (W)=2\dim (X)$, then $W=X\times Y$.
\end{enumerate}
\end{lem}

\begin{proof}
(i) The first result follows from the fact $W\subset  p_X(W)\times p_Y(W)$. 
For the second, denote by $W_x$ the fiber of $p_X|_W\colon W\to p_X(W)$ over a point $x\in p_X(W)$.
Then, $\dim (Y)\le \dim (W_x)\le \dim (\Phi (\mc{O}_x))$, and hence $\Supp (\Phi (\mc{O}_x))=Y$. Then Lemma \ref{lem:fundamental_FM} (iii) 
completes the proof.

(ii) Note that if $\dim  (\Phi (\mc{O}_x))=0$ for $x\in X$,  there is a point $y\in Y$ and an integer $n$ 
such that $\Phi (\mc{O}_x)\cong \mc{O}_{y}[n]$ 
by \cite[Lemma 4.5]{Hu06}.    We also notice that there are no other components dominating $X$, since $\Supp(\Phi (\mc{O}_x))$ is connected by Lemma \ref{lem:fundamental_FM} (ii).
Hence, for general points $x\ne x'\in X$, we have $\dim  (\Phi (\mc{O}_x))=\dim  (\Phi (\mc{O}_{x'}))=0$, and 
$\Hom^i  (\Phi (\mc{O}_x), \Phi (\mc{O}_{x'}))=0$ for all $i$. 
Then, $\Supp (\Phi (\mc{O}_x))\cap \Supp (\Phi (\mc{O}_{x'}))=\emptyset$. In particular, $W$ also dominates $Y$ by $p_Y$. 

(iii) This is obvious, since $\dim (X)= \dim (Y)$ and $W\subset X\times Y$ by definition.
\end{proof}

The equation \eqref{eqn:PQ} implies that  
Lemma \ref{lem:i-dimensional} (i) and (ii) still hold after replacing $p_X$ with $p_Y$.

For an irreducible closed subvariety $V$ of $X$, we set 
$$
\mc{C}_V:=\{C \mid C \text{ is an irreducible curve contained in $V$, satisfying } K_X\cdot C=0 \}.
$$


\begin{lem}\label{lem:kawamata}
Let $V$ be an irreducible closed subvariety of $X$, and take $W_0\in \Comp (\Phi^\mc{P}_{X\to Y})$.
\begin{enumerate}
\item
Suppose that  
\begin{equation}\label{eqn:CV}
\bigcup_{C\in \mc{C}_V}C\subsetneq V
\end{equation}
holds.
Then, we have 
\begin{equation}\label{eqn:WZX}
\dim (W_0)\le 2\dim (X)-\dim (V).
\end{equation}
\item
Suppose that  $K_X|_V$  is big. Then, the inequality \eqref{eqn:WZX}
holds. Assume furthermore that  $K_X|_V$ is nef and the equality in \eqref{eqn:WZX} holds. Then, $K_X$  is nef. 
\item
Suppose that  $-K_X|_V$ is big. Then, the inequality \eqref{eqn:WZX}
holds. Assume furthermore that  $-K_X|_V$ is nef and  the equality in \eqref{eqn:WZX} holds. Then, $-K_X$ is nef. 
\end{enumerate}
\end{lem}

\begin{proof}
(i) Set $W_{0V}:=p_X^{-1}(V)\cap W_0(\subset X\times Y)$. Then,  we have  
\begin{equation}\label{eqn:fiber_dimension}
 \dim (W_0)-\dim (X)\le \dim (W_{0V})-\dim (V).
\end{equation}
If the projection $p_Y$ contracts a curve $C'$ on $W_{0V}$,
then $p_X(C')$ is a curve on $V$. 
Denote the normalization $\widetilde{W}_{0}\to W_{0}$ by $\nu_0$, and 
take an irreducible curve $\widetilde{C'}$ on $\widetilde{W}_{0}$ 
with $\nu_0(\widetilde{C'})=C'$. Then 
$$
K_X\cdot p_X(C')=(\nu_0^*p_X^*K_X)\cdot \widetilde{C'}
=(\nu_0^*p_Y^*K_Y)\cdot \widetilde{C'}=0 
$$  
holds by Lemma \ref{lem:fundamental_FM} (iv). 
Hence,  condition \eqref{eqn:CV} implies that $p_Y|_{W_{0V}}$
is generically finite on the image, and hence $\dim (W_{0V})\le \dim (Y)$.
The result follows from the equality $\dim (X)=\dim (Y)$ and inequality \eqref{eqn:fiber_dimension}.

(ii) Take the normalization $\mu\colon \widetilde{V}\to V$. Since $\mu^*(K_X|_V)$ is also big,
Kodaira's lemma yields that $\mu^*(K_X|_V)$ is $\Q$-linearly equivalent to $A+B$, 
where $A$ is an ample $\Q$-divisor and $B$ is an effective $\Q$-divisor on $\widetilde{V}$.
Then, we have 
$$
\bigcup_{C\in \mc{C}_V}C\subset \mu (\Supp (B))\subsetneq V,
$$
 and hence,  $p_Y|_{W_{0V}}$ is generically finite on the image as in the proof of (i) and 
inequality  \eqref{eqn:WZX} holds by (i).
If equality in \eqref{eqn:WZX} holds, then inequality \eqref{eqn:fiber_dimension} implies 
$\dim (Y)\le \dim (W_{0V})$, and thus $\dim (Y)= \dim (W_{0V})$. Hence, it turns out that $p_Y|_{W_{0V}}$ is surjective. 
Since the linear equivalence
$$
((p_X\circ \nu_0)|_{\nu_0^{-1}(W_{0V})}) ^*(mK_X|_V)\sim ((p_Y\circ \nu_0)|_{\nu_0^{-1}(W_{0V})})^*(mK_Y)
$$ 
holds for some $m>0$ by Lemma \ref{lem:fundamental_FM} (iv),
$K_Y$ is nef by the assumption that $K_X|_V$ is nef . Hence, $K_X$ is nef  (see Remark \ref{rem:nef-nef}(ii)). 
The statement (iii) can be proved in a similar way.
\end{proof}

\begin{rem}\label{rem:nef-nef}
\begin{enumerate}
\item
If $K_X$ is big, i.e.~$X$ is of general type, then Lemma \ref{lem:kawamata} (ii) for $V=X$ yields $\dim (W_0)=\dim (X)$.    
In a similar way, if $-K_X$ is big, then  $\dim (W_0)=\dim (X)$ holds by Lemma \ref{lem:kawamata} (iii).    
These are actually shown by Kawamata in the proof of \cite[Theorem 2.3 (2)]{Ka02}.
\item
If $K_X$  is nef and $Y$ is a Fourier--Mukai partner of $X$, then 
$K_Y$ is nef and $\nu(X)=\nu (Y)$ holds. A similar statement is true for anticanonical divisors $-K_X$ and $-K_Y$. See \cite[Theorem 2.3]{Ka02} and \cite[Propositions 6.15, 6.18]{Hu06} for the proof.
\item
Let $\{\varphi _i\}$ be the set of all extremal contractions on $X$. Define $V$ to be a fiber of  
maximal dimension 
among all fibers of all $\varphi_i$. Then $-K_X|_V$ is ample, and hence, Lemma \ref{lem:kawamata} (iii) implies that
inequality \eqref{eqn:WZX} holds.
\end{enumerate}
\end{rem}


\begin{lem}\label{lem:numerical_iitaka_kodaira}
Let $D$ be a nef Cartier divisor and $H$ be a very ample divisor on a normal projective variety $X$.
Set $V=\bigcap _{i=1}^{\dim (X)- \nu (D)}H_i$ for general members $H_i\in |H|$. 
Then 
$$
\overline{\underset{C\subset V}{\underset{D\cdot C=0}{\bigcup}}C}\subsetneq V.
$$
\end{lem}

\begin{proof}
It follows from the definition on the numerical Iitaka-Kodaira  dimension that $D|_V$ is a nef and big divisor on $V$.
Then by  Kodaira's lemma, $D|_V$ is $\Q$-linearly equivalent to $A+B$, where $A$ is an ample $\Q$-divisor and $B$ is an effective $\Q$-divisor on $V$.
Hence, 
$\overline{\bigcup_{D\cdot C=0, C\subset V}C}\subset B$, and then  the result follows.   
\end{proof}


\begin{prop}\label{prop:kappa_codim}
Fix $W_0\in \Comp (\Phi^\mc{P}_{X\to Y})$.
\begin{enumerate}
\item 
Assume that $K_X$ is nef. 
Then, we have 
$$
\dim (W_0)\le 2\dim (X)-\nu(X).
$$
\item
Assume that $-K_X$ is nef. 
Then, we have 
$$
\dim (W_0)\le  2\dim (X)-\nu(-K_X). 
$$
\item
Assume that $\kappa (X)\ge 0$. Suppose that the minimal model conjecture and
the abundance conjecture hold. 
Then, we have 
\begin{equation}\label{eqn:kappa}
\dim (W_0)\le 2\dim (X)-\kappa(X).
\end{equation}
\end{enumerate}
\end{prop}

\begin{proof}
(i) and (ii) are direct consequences of  Lemmas \ref{lem:kawamata} (i) and \ref{lem:numerical_iitaka_kodaira}.

(iii) We may assume $\kappa (X)> 0$, since otherwise the statement is obvious.
Run the minimal model program for $X$. Then, we obtain a birational map $\phi\colon X\dashrightarrow  X_{m}$, where  
$X_{m}$ is a minimal model. 
Take a common resolution $f\colon X'\to X$ and $g\colon X'\to  X_{m}$.
Then \cite[Lemma 4.4]{Ka02} states that there is an integer $n>0$
such that $D:=n(f^*K_X-g^*K_{X_m})$ is effective. 

Let $H$ be a very ample divisor on $X_m$.
Set $V_m:=\bigcap _{i=1}^{\dim (X)- \kappa (X)}H_i$ for general members $H_i\in |H|$. Then, take its strict transform on $X$ and denote it by $V$.
We see that  $V$  satisfies the condition \eqref{eqn:CV}. Indeed, assume
\begin{equation}\label{eqn:CVV}
\bigcup_{C\in \mc{C}_V}C=V
\end{equation}
 for a contradiction.
Let us set
$$
\widetilde{\mc{C}_{V}}:=\{  C\in \mc{C}_V \mid C\cap U\ne \emptyset \},
$$
where 
$U$ is the open subset of $X$ on which $\phi$ is an isomorphism.
Then we know 
\begin{equation}\label{eqn:VmU}
\emptyset \ne V_m\cap \phi(U)\subset\bigcup_{C\in \widetilde{\mc{C}_{V}}}\phi_*C
\end{equation}
by \eqref{eqn:CVV} and the choice of $V_m$. 
Take irreducible curves  $C\in \widetilde{\mc{C}_{V}}$ and $C'$ on $X'$ satisfying $f(C')=C$.
Then, we have
$$
0=nK_X\cdot C=nf^*K_X\cdot C'= ng^*K_{X_m}\cdot C'+D\cdot C'.
$$ 
On the other hand, we have $g^*K_{X_m}\cdot C'\ge 0$ and $D\cdot C'\ge 0$, since $K_{X_m}$ is nef, and $f(\Supp (D))\cap U=\emptyset$.
Therefore, we have 
$g^*K_{X_m}\cdot C'=D\cdot C'=0$, and thus
$$
K_{X_m}\cdot \phi_*C=K_{X_m}\cdot g(C')=0.
$$ 
In particular, we know that
$$
\{ \phi _*C \mid C\in \widetilde{\mc{C}_{V}}\} \subset \mc{C}_{V_m},
$$
and hence
$$
V_m  \subset \overline{\bigcup_{C\in \widetilde{\mc{C}_{V}}}\phi_*C}\subset \overline{\bigcup_{C_m\in \mc{C}_{V_m}}C_m}.
$$
Here, the first inclusion follows from \eqref{eqn:VmU}.
This produces a contradiction to 
$$
\overline{\bigcup_{C_m\in \mc{C}_{V_m}}C_m}\subsetneq V_m 
$$
obtained by Lemma \ref{lem:numerical_iitaka_kodaira} and $\nu (X_m)=\kappa (X_m)=\kappa (X)$.
Therefore, the result follows from Lemma \ref{lem:kawamata} (i).
\end{proof}


\begin{cor}\label{cor:maximum}
Take $W_0\in \Comp (\Phi^\mc{P}_{X\to Y})$.
\begin{enumerate}
\item
If $\dim (W_0)=2\dim (X)$, then we have $K_X\equiv 0$.
\item
If $\dim (W_0)=2\dim (X)-1$ and $K_X\not\equiv 0$, then either $K_X$ is nef and  $\nu(X)=1$
 or else  $-K_X$ is nef and $\nu(-K_X)=1$. 
\end{enumerate}
\end{cor}

\begin{proof}
(i) We can see $\Supp (\Phi^\mc{P}_{X\to Y}(\mc{O}_x))=Y$ by Lemma \ref{lem:i-dimensional} (iii). Then Lemma \ref{lem:fundamental_FM} (iii) implies the statement. 
 
(ii)  Lemma \ref{lem:kawamata} (ii) and (iii) yield that $K_X$ or $-K_X$ is nef. 
Thus, Proposition \ref{prop:kappa_codim} (i) and (ii) imply the result. 
\end{proof}


\begin{rem}
Suppose that  $\kappa (X)=0$ or $1$. Then Corollary \ref{cor:maximum} implies that equality in \eqref{eqn:kappa} cannot be attained unless $X$ is a minimal model. On the other hand, 
in the case $\kappa(X)\ge 2$, 
 equality in \eqref{eqn:kappa} may possibly hold for a non-minimal model $X$ as follows.

Set $X:=S\times \widehat{E}$ and  $Y:=S\times E$ for
an elliptic curve $E$ and a smooth projective surface $S$ of general type. 
Then it satisfies $\kappa (X)=2$. 
Consider a Poincar\'{e} bundle $\mc{P}_E$ on $\widehat{E}\times E$. If 
$h\colon \Delta_S\times \widehat{E}\times E\to  \widehat{E}\times E$ denotes the projection,  then $h^*\mc{P}_E\in D(X\times Y)$ gives rise to  
a Fourier--Mukai transform between $D(X)$ and $D(Y)$ (see \cite[Exercise 5.20]{Hu06}).  
Assume furthermore that $S$ is not minimal. Then $X$ is not minimal, but  equality in
 \eqref{eqn:kappa} holds.
\end{rem}

\section{Fourier--Mukai support dimensions}
Let $X$ and $Y$ be smooth projective varieties, and
consider  a Fourier--Mukai transform
$$
\Phi=\Phi^{\mathcal{P}}_{X\to Y}\colon D(X)\to D(Y).
$$ 
We give names to some special types of Fourier--Mukai transforms:

\begin{defn}\label{defn:diagonal}
A Fourier--Mukai transform $\Phi^\mc{P}_{X\to Y}$ is said to be \emph{$K$-equivalent type},
if there is an element $W_0\in \Comp (\Phi^\mc{P}_{X\to Y})$ such that $\dim (W_0)=\dim (X)$.
Similarly, it is said to be \emph{Calabi--Yau type},
if there is an element $W_0\in \Comp (\Phi^\mc{P}_{X\to Y})$ such that $\dim (W_0)=2\dim (X)$. 
\end{defn}

Note that in both cases, it turns out that the set  $\Comp (\Phi^\mc{P}_{X\to Y})$ consists of the unique element by Lemma \ref{lem:i-dimensional} (ii) and (iii).

\begin{exa}\label{ex:kernel_support}
\begin{enumerate}
\item
Any standard autoequivalences are $K$-equivalent type by the description of the kernel object in 
\eqref{eqn:standard}. 
\item
Let $\alpha$ be a spherical object in $D(X)$. Note that 
$\RHom_{D(X)}(\alpha,\mc{O}_x)=0$ if and only if $x\not\in \Supp (\alpha)$
by \cite[Lemma 4.2]{BM02}.
Then, by the triangle \eqref{eqn:twist_triangle} for $\beta=\mc{O}_x$, 
we see that 
$$
\Supp (T_\alpha(\mc{O}_x))=\Supp (\alpha)
$$
for $x\in \Supp (\alpha)$, and 
$$T_\alpha(\mc{O}_x)=\mc{O}_x$$ 
for  $x\not\in \Supp (\alpha)$.
Consequently, equations \eqref{eqn:LiP} imply 
$$
\Supp (\mc{C})=\Delta_X\cup (\Supp (\alpha) \times  \Supp (\alpha))(\subset X\times X),
$$
where $\mc{C}\in D(X\times X)$ is the kernel object of $T_\alpha$, given in \eqref{eqn:kernel_twist}.

Let $C$ be a $(-2)$-curve on a smooth projective surface $X$.  
Then, the twist functor $T_{\mc{O}_C}$ is $K$-equivalent type. On the other hand,
the twist functor $T_{\mc{O}_X}$ along the structure sheaf $\mc{O}_X$ on a K3 surface 
$X$ is Calabi--Yau type. 
\end{enumerate}
\end{exa}
 
\begin{rem}\label{rem:diagonal}
\begin{enumerate}
\item
Let  $\Phi^\mc{P}_{X\to Y}$ be a  Fourier--Mukai transform  of $K$-equivalent type, 
and take the unique element  $W_0\in \Comp (\Phi^\mc{P}_{X\to Y})$.
Then, Kawamata shows in  the proof of \cite[Theorem 2.3 (2)]{Ka02} that 
$p_X|_{W_0}$ and $p_Y|_{W_0}$ are birational morphisms, and that  
$W_0$ is the graph of the birational map $(p_Y|_{W_0})\circ (p_X|_{W_0})^{-1}$ between $X$ and $Y$.
Moreover, if we take a resolution of singularities $f\colon Z\to W_0$,
then  the linear equivalence $f^*(p_X|_{W_0})^*(K_X)\sim f^*(p_Y|_{W_0})^*(K_Y)$ holds 
(use Lemma \ref{lem:fundamental_FM} (iv), and we can take $m=1$ on the resolution $Z$. See \cite[Proposition 6.19]{Hu06}). 
In other words, varieties $X$ and $Y$ are \emph{K-equivalent}.
\item 
For a given Fourier--Mukai transform $\Phi^\mc{P}_{X\to Y}$, 
if $\dim (\Phi^\mc{P}_{X\to Y}(\mc{O}_x))=0$ for a point $x\in X$, then $\Phi^\mc{P}_{X\to Y}(\mc{O}_x)=\mc{O}_y[i]$ for some point $y\in Y$ and $i\in \Z$
by \cite[Lemma 4.5]{Hu06}.
Moreover, a Fourier--Mukai transform $\Phi^\mc{P}_{X\to Y}$ is a $K$-equivalent type if and only if $\dim (\Phi^\mc{P}_{X\to Y}(\mc{O}_x))=0$ holds for a general point $x\in X$. Consequently, we can see that a composition of  Fourier--Mukai transforms of $K$-equivalent type is again $K$-equivalent type.
\end{enumerate}
\end{rem}

Next let us consider the case $X=Y$, i.e. $\Phi=\Phi^{\mc{P}}\in \Auteq D(X)$. Then we define
$$
N_X:=\max \{\dim (W_0) \mid W_0\in \Comp (\Phi^\mc{P}) \text{ for some }\Phi^{\mc{P}}\in \Auteq D(X) \}
$$
and call it by the \emph{Fourier--Mukai support dimension} of $X$.
Obviously, we have 
$$
\dim (X)\le N_X\le 2\dim (X). 
$$
Let us consider two extreme cases below; the case $N_X=\dim (X)$ and the case $N_X= 2\dim (X)$.

\subsection{$K$-equivalent type}
Define
$$
\Auteq_{\Kequiv}D(X):=\{\Phi \in \Auteq D(X) \mid \Phi \text{ is $K$-equivalent type}\}.
$$
Then, Remark \ref{rem:diagonal}(ii) tells us that $\Auteq_{\}}D(X)$ is a subgroup of  $\Auteq D(X)$.
It is easy to see by definition and Remark \ref{rem:diagonal}(i) that the following conditions are equivalent;
\begin{itemize}
\item
$\Auteq_{\Kequiv}D(X)=\Auteq D(X)$.
\item
$N_X=\dim (X)$.
\item
For any $\Phi^\mc{P}\in \Auteq D(X)$, the set $\Comp (\Phi^\mc{P}{})$ consists of the unique element $W_0(\subset X\times X)$, which is 
the graph of a birational automorphism of $X$.
\end{itemize}
If one, and hence, all of these conditions are satisfied, the autoequivalence group $\Auteq D(X)$ (or, simply $X$) is said to be  \emph{$K$-equivalent type}.


\begin{prop}[Kawamata]\label{prop:BO}
Let $X$ a smooth projective variety with $\pm K_X$ big.
Then  $\Auteq D(X)$ is $K$-equivalent type.
\end{prop}

\begin{proof}
It follows from Remark \ref{rem:nef-nef} (i). 
\end{proof}

\subsection{Calabi--Yau type}
By definition, the following conditions are equivalent;
\begin{itemize}
\item
There is an autoequivalence  $\Phi$ of $D(X)$ such that $\Phi$ is Calabi--Yau type.
\item
$N_X=2\dim (X)$.
\end{itemize}
If one, and hence, two of these conditions are satisfied, the autoequivalence group $\Auteq D(X)$ (or, simply $X$) is said to be \emph{Calabi--Yau type}.
Note that Corollary \ref{cor:maximum} yields $K_X\equiv 0$ in this case.
It is natural to ask whether the converse is true or not.

\begin{prob}\label{prob:CY}
Suppose that $K_X\equiv 0$. Then, is $\Auteq D(X)$ Calabi--Yau type?
\end{prob}

We give an affirmative answer to Problem \ref{prob:CY} for abelian varieties in Proposition
 \ref{prop:CY}, for curves $X$ in  Theorem \ref{thm:dichotomy}, and 
for surfaces $X$ in Theorem \ref{thm:trichotomy}.


\begin{prop}\label{prop:CY}
Let $X$ be an abelian variety.
Then,  $\Auteq D(X)$ is  Calabi--Yau type.
\end{prop}

\begin{proof} 
Let us put $d:=\dim (X)$, and consider the normalizes Poincar\'e bundle $\mc{P}$ on $X\times \widehat{X}$.
Then, the integral functor $\Phi^{\mc{P}}_{X\to \widehat{X}}$ 
is an equivalence (cf.~\cite[Proposition 9.19]{Hu06}), and 
the cohomological Fourier--Mukai transform  $\Phi^{\mc{P},H}_{X\to \widehat{X}}$ induces an isomorphism between
the total cohomologies $H^*(X,\Q)$ and $H^{*}(\widehat{X},\Q)$,
which restricts an isomorphism  
$H^n(X,\Q)$ and $H^{2d-n}(\widehat{X},\Q)$ for any $n$.
The last isomorphism coincides with
\begin{equation}\label{eqn:PD}
(-1)^{\frac{n(n+1)}{2}+d}\cdot \text{PD}_n\colon H^n(X,\Q)\to H^{2d-n}(\widehat{X},\Q)\cong H^{2d-n}(X,\Q)^*,
\end{equation}
where $\text{PD}_n$ is Poincar\'e duality
(see \cite[Lemma 9.23]{Hu06}).
Here note that there is a natural isomorphism between $H^{2d-n}(\widehat{X},\Q)$ and 
$H^{2d-n}(X,\Q)^*$ by the construction of the dual abelian variety $\widehat{X}$. 

For an ample line bundle $\mathcal{L}$ on $\widehat{X}$,  consider the Fourier--Mukai transform 
$$
\Phi^\mc{Q}:=\Phi^{\mc{P}}_{\widehat{X}\to X}\circ ((-)\otimes \mathcal{L})\circ  \Phi^{\mc{P}}_{X\to \widehat{X}}\in \Auteq D(X)
$$
with a kernel object $\mc{Q}\in D(X\times X)$.
The cohomological Fourier--Mukai transform induced by the autoequivalence $(-)\otimes \mc{L}$ of $D(\widehat{X})$ is just multiplying by 
$\ch (\mathcal{L})$. For a point $x\in X$, 
we have $\ch (\mathcal{O}_x)= (0,\ldots,0, 1)\in H^{2*}(X,\Q)$,
and hence \eqref{eqn:PD} yields
\begin{align*}
\Phi^{\mc{Q},H} (\ch (\mathcal{O}_x))
&= \Phi^{\mathcal{P},H}_{\widehat{X}\to X}( \ch (\mathcal{L})\cdot \Phi^{\mathcal{P},H}_{X\to\widehat{X}} ((0,\ldots, 0,1)))\\
&= \Phi^{\mathcal{P},H}_{\widehat{X}\to X}(\ch (\mathcal{L})\cdot (1,0,\ldots, 0))\\
&= \Phi^{\mathcal{P},H}_{\widehat{X}\to X}(\ch (\mathcal{L}))\\
&= (\ch_d (\mc{L}),-\ch_{d-1} (\mc{L}),\ch_{d-2}(\mc{L}),\ldots,(-1)^{d}\ch_0 (\mc{L})).
\end{align*}
Therefore, the $0$-th cohomology component of $\Phi^{\mc{Q},H} (\ch (\mathcal{O}_x))$ is 
$\ch_{d} (\mathcal{L})=\frac{1}{d!}c_1(\mathcal{L})^{d}$, which is not $0$. This means $\Supp (\Phi^{\mc{Q}} (\mc{O}_x))=X$, and hence 
$\Supp \mc{Q}=X\times X$ (see the equations \eqref{eqn:LiP}). In particular, we obtain $N_X=2\dim (X)$.
\end{proof}

Now, we can show  a dichotomy of the autoequivalence groups of smooth projective curves.
 

\begin{thm}[Dichotomy]\label{thm:dichotomy}
Let $C$ be a smooth projective curve with the genus $g(C)$, and $N_C\in\{1,2\}$ be its Fourier--Mukai support dimension. 
\begin{enumerate}
\item
$N_C=2$  (Calabi--Yau type) if and only if $g(C)=1$, namely $C$ is an elliptic curve.
\item
$N_C=1$ ($K$-equivalent type) if and only if $g(C)\ne 1$, namely $C$ is a projective line or a curve of general type.
\end{enumerate}
\end{thm}
 
\begin{proof}
If $C$ is not an elliptic curve, then $\pm K_C$ is ample. 
Hence, Proposition \ref{prop:BO} tells us that $N_C=1$.
Since  elliptic curves  are $1$-dimensional abelian varieties,  Proposition \ref{prop:CY} completes the proof. 
\end{proof}

Theorem \ref{thm:dichotomy} shows that  Fourier--Mukai support dimensions
 of smooth projective curves reflect their geometry. We obtain a similar result for smooth projective surfaces in Theorem \ref{thm:trichotomy}.


\section{Trichotomy of autoequivalence groups  on  smooth projective surfaces}\label{sec:trichotomy}
In this section we show a similar result in the $2$-dimensional case to Theorem \ref{thm:dichotomy}.


\begin{lem}\label{lem:Gamma_W}
Let $S$ be  smooth projective surface, and 
take $\Phi^\mc{P}\in \Auteq D(S)$ and $W_0\in \Comp (\mc{P})$. 
Then we have $\dim (\mc{P})=\dim (W_0)$. In particular, 
$$
N_S=\max \{ \dim (\mc{P}) \mid \Phi^\mc{P}\in \Auteq D(S)\}
$$
 holds.
\end{lem}

\begin{proof} 
Set $\Gamma:=\Supp (\mc{P})$. 
Note that  $2\le \dim (W_0)\le \dim (\Gamma)\le 4$.
Obviously,  $\dim (\Gamma)=4$ is equivalent to $\dim (W_0)=4$. 

Suppose that $\dim (W_0)=2$. Then Lemma \ref{lem:i-dimensional} (ii) implies that $W_0$ is the unique irreducible component of $\Gamma$ dominating $S$ by $p_1$,
 and is also the unique irreducible component dominating $S$ by $p_2$.
Hence, Lemma \ref{lem:i-dimensional} (i) forces that there are no $3$-dimensional irreducible components.
In particular,  $\dim (\Gamma)=2$. 
To the contrary, if $\dim (\Gamma) =2$, then $\dim (W_0)=2$ follows. 
This completes the proof.
%
\end{proof}

\begin{rem}
Let $X$ be a Calabi--Yau $3$-fold, i.e. it satisfies $\omega_X\cong\mathcal{O}_X$ and $H^1(X,\mathcal{O}_X)=0$,  and suppose that $X$ contains $E\cong \PP^2$. Note that the normal bundle $\mc{N}_{E/X}$ is isomorphic to $\mc{O}_{\PP^2}(-3)$.
Then, we can see that $\mc{O}_E$ is a spherical object of $D(X)$ (We leave the proof of this fact to readers. Use \cite[Proposition 11.8]{Hu06} and the Local-to-Global Ext spectral sequence). 
The kernel of the twist functor $T_{\mc{O}_E}$ has two irreducible components. One is  supported on the diagonal $\Delta_X$ in $X\times X$ and the other one  on $E\times E$
(see Example \ref{ex:kernel_support} (ii)).  
Hence, Lemma \ref{lem:Gamma_W} is false for higher dimensional varieties. 
\end{rem}

Now, we are in a position to show a trichotomy of autoequivalence groups on smooth projective surfaces.


\begin{thm}[Trichotomy]\label{thm:trichotomy}
Let $S$ be a smooth projective surface and $N_S\in \{2,3,4\}$ be the Fourier--Mukai support dimension of $S$.
\begin{enumerate}
\item
$N_S=4$  (Calabi--Yau type) if and only if $K_S\equiv 0$.
\item
$N_S=3$ if and only if $S$ has a minimal elliptic fibration and $K_S\not \equiv 0$. 
\item
$N_S=2$  ($K$-equivalent type) if and only if $S$ has no minimal elliptic fibration and $K_S\not \equiv 0$.
\end{enumerate}
\end{thm}

\begin{proof}
(i) For each surface $S$ with $K_S\equiv 0$, let us give an example of auto\-equivalence whose kernel object has  $4$-dimensional support. 

First, take a K3 surface $S$ and let $\mc{P}$ be the ideal sheaf $I_{\Delta_S}$ of the diagonal 
$\Delta_S$ in $S\times S$. For $x\in S$, the integral functor $\Phi^{\mc{P}}$ satisfies $\Phi^{\mc{P}}(\mc{O}_x)=I_x$, the ideal sheaf of the point $x$, and then \cite[Corollary 2.8]{BM01} implies that  $\Phi^{\mc{P}}$ is an autoequivalence.

For an abelian surface $S$, we have already shown $N_S=4$ in Proposition \ref{prop:CY}.

Take an Enriques surface $T$. Then there is a K3 surface $S$ with an involution $\iota$ on
 $S$ such that $T$ is the quotient of $S$ by $\Span {\iota}\cong \Z/2\Z$. 
Then it turns out that the autoequivalence  $\Phi^{\mc{P}}$ given above for a K3 surface $S$
descends to an autoequivalence of $D(T)$, and its kernel  has a $4$-dimensional support.    
See \cite[Example 5.2]{BM98} for details.

For a bielliptic surface $S$, there are elliptic curves $E_1, E_2$ and a finite group $G$ acting diagonally on 
$E_1\times E_2$ 
 such that $S=(E_1\times E_2)/G$. 
Therefore, $S$ has two minimal elliptic fibrations $\pi_i\colon S\to E_i/G$. 
Take a universal sheaf $\mc{U}_i$ on $J_{S/(E_i/G)}(1)\times_{E_i/G}S$ ($i=1,2$) given in \S \ref{subsec:bridgeland}. Fix an isomorphism between $J_{S/(E_i/G)}(1)$ and $S$, and 
regard $\Phi_i:=\Phi_{J_{S/(E_i/G)}(1)\to S}^{\mc{U}_i}$ as an autoequivalence of $D(S)$.
Then the kernel of the composition $\Phi_1\circ \Phi_2$ is $4$-dimensional.

Conversely, it follows from Corollary \ref{cor:maximum} (i) that the equality $N_S=4$ implies  the equality
$K_S\equiv 0$.

(ii)
 First note that the equality $N_S=3$ implies that either $K_S$ is nef and  $\nu(S)=1$, or that  $-K_S$ is nef and $\nu(-K_S)=1$ by Corollary \ref{cor:maximum} and Lemma \ref{lem:Gamma_W}.
Moreover, if $K_S$ is nef, it is known that $S$ has a minimal elliptic fibration (cf.~\cite[Proposition~IX.2]{Be96}). Therefore, we consider the only case $-K_S$ is nef and $\nu(-K_S)=1$. Note that in this case, there is a smooth rational curve $C$ on $S$ with $K_S\cdot C< 0$.

Take an autoequivalence $\Phi=\Phi^\mc{P}$ of $D(S)$ with $\dim (\mc{P})=3$. Then Lemma \ref{lem:Gamma_W} implies that there is a $3$-dimensional irreducible component $W_0$ of $\Supp (\mc{P})$ dominating $S$ by $p_1$.
Let us denote by  $W_{0x}(\subset \{x \}\times S)$ the fiber of the morphism $p_1|_{W_{0}}\colon W_0\to S$ over a point $x\in S$, and regard it as a divisor on $S$ by the isomorphism 
$\{x \}\times S\cong S$. If $\dim W_{0x}=2$ for some $x$, 
Lemma \ref{lem:fundamental_FM} (iii) supplies  a  contradiction to $K_S\not\equiv 0$. 
Hence, every fiber of  $p_1|_{W_{0}}$ is $1$-dimensional, and therefore  $p_1|_{W_{0x}}$ is flat (cf.~\cite[Exercise III.10.9]{Ha77}).
Take points $x,y\in C$. Since $C$ is isomorphic to $\mathbb{P}^1$, the points $x$ and $y$ are rationally equivalent $0$-cycles on $S$.
Hence, the divisors $W_{0x}$ and $W_{0y}$ on $S$ are linearly equivalent (see \cite[Theorems 1.1.4, 1.1.7]{Fu98}).

If $\bigcap _{x\in C}W_{0x}\ne \emptyset$, then we see 
$C\subset \Supp (\Phi^{-1} (\mc{O}_z))$ for a point $z\in \bigcap _{x\in C}W_{0x}$.
This contradicts $K_S\cdot C\ne 0$, since $\Phi^{-1} (\mc{O}_z)$ is a Calabi--Yau object. 
Hence, we conclude $\bigcap _{x\in C}W_{0x}= \emptyset$,
 and therefore the complete linear system 
$\mathfrak{d}:=|W_{0x}|$ is base point free. 
Moreover, note that  $K_S\cdot W_{0x}=0$ for each  $x\in C$, since $W_{0x}$ is contained in $\Supp (\Phi (\mc{O}_x))$.
Furthermore, the Hodge index theorem implies $W_{0x}\cdot W_{0x}=0$. 
Then we can see that  $\mathfrak{d}$ defines a minimal elliptic fibration, after taking the Stein factorization if necessary.

Conversely, if $S$ has a minimal elliptic fibration,   take a universal sheaf $\mathcal{U}$ on 
$J_S(1)\times S$. Then $\Phi^\mathcal{U}_{J_S(1)\to S}$ is a Fourier--Mukai transform.
Since $J_S(1)\cong S$ and $\dim (\mathcal{U})=3$, we obtain $N_S=3$. 

(iii) The result follows from (i) and (ii).
\end{proof}


For a smooth projective curve $C$, every Fourier--Mukai partner of $C$ is isomorphic to $C$. Therefore, 
it is obvious that Fourier--Mukai support dimension is a derived invariant for smooth projective curves.
In the surface case, a similar result holds.

\begin{cor}
Fourier--Mukai support dimension is a derived invariant for smooth projective surfaces, i.e.
if smooth projective surfaces $S$ and $T$ are Fourier--Mukai partners, then $N_S=N_T$.
\end{cor}

\begin{proof}
Let $T$ be a Fourier--Mukai partner of a smooth projective surface $S$.
Suppose that $T$ is not isomorphic to $S$. Then
$S$ is either a K3 surface, an abelian surface or a minimal elliptic surface, and
moreover, $T$ is also a surface of the same type as $S$ (\cite{BM01, Ka02}). Therefore, 
Theorem \ref{thm:trichotomy} implies the conclusion. 
\end{proof}

\begin{conj}\label{conj:invariant}
\begin{enumerate}
\item
Fourier--Mukai support dimension is a derived invariant for smooth projective varieties, i.e.
if  smooth projective varieties $X$ and $Y$ are Fourier--Mukai partners, then $N_X=N_Y$.
\item
Let $X$ and $Y$ be a smooth projective varieties. 
Assume that $X$ is of $K$-equivalent type and that there is a Fourier--Mukai transform $\Phi^\mc{P}_{X\to Y}$.
Then,  $\Phi^\mc{P}_{X\to Y}$ is $K$-equivalent type. 
\end{enumerate}
\end{conj}

\begin{rem}
\begin{enumerate}
\item Conjecture \ref{conj:invariant} (ii) implies Conjecture \ref{conj:invariant} (i)  for smooth projective varieties of $K$-equivalent type.
In fact, suppose that Conjecture \ref{conj:invariant} (ii) is true. Then $Y$ in  Conjecture \ref{conj:invariant} (ii)  is also $K$-equivalent type, since
a composition of  Fourier--Mukai transforms of $K$-equivalent type is  $K$-equivalent type.
In particular, Conjecture \ref{conj:invariant} (i) is true for $X$ of $K$-equivalent type.
\item
Kawamata predicts in \cite[Conjecture 1.2]{Ka02} that birationally equivalent, derived equivalent smooth projective varieties are $K$-equivalent,
but a counterexample to his conjecture is discovered by the author in \cite{Ue04}.
 Conjecture \ref{conj:invariant} (ii) is a special version of Kawamata's conjecture, since $X$ and $Y$ are 
$K$-equivalent by Remark \ref{rem:diagonal} (i) when a Fourier--Mukai transform $\Phi^\mc{P}_{X\to Y}$ is $K$-equivalent type. 
\end{enumerate}
\end{rem}


\section{Autoequivalence groups of $K$-equivalent type}
Let $S$ be  smooth projective surface and take $\Phi=\Phi^\mc{P}\in \Auteq_{\Kequiv} D(S)$.
Set $$\Gamma:=\Supp (\mathcal{P}).$$ 
Then $\dim \Gamma=2$ by Lemma \ref{lem:Gamma_W}, and Lemma \ref{lem:i-dimensional} (ii) yields that there is the unique 
 component $W_0$ of $\Gamma$ dominating $S$ by both of $p_1$ and $p_2$. Note that $\Gamma_x$ is at most $1$-dimensional for $x\in S$.

Let us denote by $Z$  the union of $(-2)$-curves on $S$.
The set $Z$ has finitely many connected components, since the Picard number $\rho(S)$ is finite. But it can possibly have infinitely many irreducible components.
If a K3 surface $S$ contains $(-2)$-curves, and $S$ admits the infinite automorphism group, then the set $Z$ on $S$ is an example of such.  

We first show Proposition \ref{prop:S_to_Z} below.
We need several claims to prove it.
Take a $(-2)$-curve $C$ on $S$ and $\mc{L}\in \Pic (C)$.
We regard $\mc{L}$ as an object of $D(S)$ in a natural way.


\begin{cla}\label{cla:pure_rigid}
We have $\dim (\Phi (\mc{L}))=1$. Moreover, every cohomology sheaf $\mc{H}^i(\Phi(\mc{L}))$ 
is rigid and pure $1$-dimensional.
\end{cla}

\begin{proof}
Note that  $\Supp(\Phi (\mc{L}))$ is contained in $p_2(p_1^{-1}(C)\cap \Gamma)$.
Then, we see  $\dim (\Phi (\mc{L}))\le 1$, since $W_0$ is the unique component dominating $S$ by $p_2$.
Since $\mc{L}$ is rigid on $S$, but $\mc{O}_x$ is not rigid for any $x\in S$,
we have $\dim (\Phi (\mc{L}))>0$ 
by \cite[Lemma 4.5]{Hu06}.
Moreover, \cite[Proposition 3.5]{IU05} implies that $\mc{H}^i(\Phi(\mc{L}))$ is rigid, and  
it is pure $1$-dimensional by \cite[Lemma 3.9]{IU05}.
\end{proof}


\begin{cla}\label{cla:supp_L}
We have $\Supp (\Phi(\mc{L}))\subset Z$.
\end{cla}

\begin{proof}
Take an irreducible component $D$ of $\Supp (\Phi(\mc{L}))$. 
Then, we see $D\cdot K_S=0$, since $\Phi(\mc{L})$ is a Calabi--Yau object.

Take an integer $i$ such that $\Supp (\mc{H}^i(\Phi(\mc{L})))$ contains the irreducible curve $D$.  
Let us set 
$$
\mc{M} :=\mc{H}^i(\Phi(\mc{L}))
$$
and 
$$
\Supp (\mc{M})=E\cup D,
$$
where the closed subset $E$ does not contain $D$.
Then consider the short exact sequence 
$$
0\to \mc{H}^0_{E}(\mc{M}) \to \mc{M}\stackrel{\phi}{\to} \mc{K}\to 0
$$
in $\Coh (S)$, where $\mc{H}^0_{E}(\mc{M})$ is the subsheaf with supports in $E$ (cf. \cite[Exercise II.1.20]{Ha77}).  

Note that $\Supp (\mc{K})=D$ and hence $\dim (\mc{H}^0_{E}(\mc{M})) \cap D\le 0$. 
Assume for a contradiction that $\mc{K}$ is not pure $1$-dimensional.
Then, there is a local section $s$ of $\mc{K}$ such that $s(x)\ne 0$ for some point $x\in S$, 
but $s(y)=0$ for all point $y\in S$ except $x$.
Let $t$ be a local section of $\mc{M}$ which is a lift of $s$. If $x\not\in E$, then $\phi$ is an isomorphism around the point 
$x$, and hence $t$ generates a $0$-dimensional subsheaf of $\mc{M}$,
which contradicts Claim \ref{cla:pure_rigid}.
Suppose that $x\in E$. Then, $t$ gives a local section of  $\mc{H}^0_{E}(\mc{M})$, and hence
$s=\phi(t)$ should be $0$, which also gives a contradiction.
Therefore we can conclude that $\mc{K}$ is pure $1$-dimensional.
 Thus,  Serre duality yields  
$$
\Ext^2_S( \mc{K}, \mc{H}^0_{E}(\mc{M}) )=\Hom_S (\mc{H}^0_{E}(\mc{M}) , \mc{K})^\vee=0,
$$
and then, \cite[Lemma 2.2 (2)]{KO95} implies that
$\mc{K}$ is rigid. Therefore, we have
\begin{align*}
2\le &\dim \Hom_S(\mc{K},\mc{K})+\dim \Ext ^2_S(\mc{K},\mc{K})\\
 =&\chi(\mc{K},\mc{K})=-c_1(\mc{K})\cdot c_1(\mc{K}).
\end{align*}
Consequently, we have $D^2<0$ and hence, $D$ is a $(-2)$-curve.
\end{proof}


\begin{cla}\label{cla:Z_to_Z}
$\Supp (\Phi(\mc{O}_x))$ is contained in $Z$ for any point  $x\in Z$. 
\end{cla}

\begin{proof}
Take a $(-2)$-curve $C$ containing $x$. 
Then, there is an exact triangle
$$
\Phi(\mc{O}_C(-1))\to\Phi(\mc{O}_C)\to\Phi(\mc{O}_x),
$$ 
which implies 
$$
\Supp(\Phi(\mc{O}_x)) \subset \Supp (\Phi(\mc{O}_C(-1)))\cup \Supp (\Phi(\mc{O}_C)).
$$
This completes the proof by Claim \ref{cla:supp_L}.
\end{proof}

\begin{prop}\label{prop:S_to_Z}
Let $S$ be a smooth projective surface.
Then, there is a group homomorphism 
$$
\iota_Z\colon \Auteq_{\Kequiv} D(S)\to \Auteq D_Z(S) \quad \Phi\mapsto \Phi|_{D_Z(S)}.
$$
\end{prop}
\begin{proof}
Claim \ref{cla:Z_to_Z} and \cite[Lemma 2.4]{Ue16} complete the proof.
\end{proof}

We define the group $\Br_Z(S)$ generated by twist functors along spherical objects supported in $Z$;
$$
\Br_Z(S)=\Span{T_{\alpha}\mid \alpha \in D_Z(S) \text{ spherical object}}(\subset \Auteq_{\Kequiv} D(S)).
$$
The following is crucial to show Theorems \ref{thm:autoeq_nontorsion} and \ref{thm:generation}.

\begin{lem}\label{lem:sheaf}
Let $S$ be a smooth projective surface.
\begin{enumerate}
\item
For a point $x\in S$ and $\Phi\in \Auteq D(S)$, 
suppose that $\Supp (\Phi(\mc{O}_x))$ is at most $1$-dimensional, and contains no $(-2)$-curves. Then,
 $\Phi(\mc{O}_x)$ is a shift of a sheaf.
\item 
Suppose that $\Phi\in \Auteq D_{Z}(S)$ preserves
the cohomology class $\ch(\mc{O}_x) \in H^4(S,\Q)$ for all points $x\in Z$.
Then, there is  an autoequivalence $\Psi\in \Br_Z(S)$, an integer  $i$ and a point $y\in Z$ 
satisfying $\Psi\circ\Phi (\mc{O}_x)\cong \mc{O}_y[i]$.  
\end{enumerate}
\end{lem}

\begin{proof}
(i)
Recall that every $1$-dimensional component $C$ of 
$\Supp (\mc{H}^i(\Phi(\mathcal{O}_x)))$ satisfies $K_S\cdot C=0$ and  
is not a $(-2)$-curve, and then  we obtain $C^2\ge 0$. 
Thus the Riemann--Roch theorem yields 
$$
\chi(\mc{H}^i(\Phi(\mathcal{O}_x)),\mc{H}^i(\Phi(\mathcal{O}_x)))=-c_1(\mc{H}^i(\Phi(\mathcal{O}_x)))\cdot c_1(\mc{H}^i(\Phi(\mathcal{O}_x))) \le 0
$$
(see \cite[\S 2.2]{Ue16}). In particular, 
$$
\dim \Ext^1_{S}(\mc{H}^i(\Phi(\mathcal{O}_x)),\mc{H}^i(\Phi(\mathcal{O}_x)))\ge 2
$$
holds. Now the statement follows from \cite[Lemma 2.9]{BM01}. 

(ii)
Notice that the proof of \cite[Key proposition]{IU05} works in our situation, and that the assumption  
$\Phi^H$ preserves the class $\ch (\mc{O}_x)$  is needed in the proof of  \cite[Condition 7.5]{IU05} (see also \cite[footnote in pp.~572]{Ue16}).
Then the result follows.
\end{proof}

Let us define 
\begin{align*}
\Auteq^\dagger_{\Kequiv} D_Z(S):&=\Image \iota_Z\\
\Aut_Z(S):&=\{\varphi \in \Aut (S) \mid \varphi|_{Z}={\id}_Z \}\\
\Pic_Z(S):&=\{\mc{L}\in \Pic (S) \mid \mc{L}|_Z\cong \mc{O}_Z \}.
\end{align*}



\begin{thm}\label{thm:autoeq_nontorsion}
Let $S$ be a smooth projective surface. Then, there is a short exact sequence
\begin{align*}
1\to\Pic _Z(S)\rtimes \Aut_Z (S) \to \Auteq_{\Kequiv} D(S) 
\xrightarrow{\iota_Z} \Auteq^\dagger_{\Kequiv} D_Z(S) \to 1.
\end{align*}
\end{thm}

\begin{proof}
Take $\Phi=\Phi^{\mc{P}}\in \Auteq_{\Kequiv} D(S)$ and a point $x\in S\backslash Z$. 
Then, we know by Claim \ref{cla:Z_to_Z} that $\Supp (\Phi(\mc{O}_x))\cap Z=\emptyset$ as in \cite[Corollary 3.5]{Ue16}.
Moreover Lemma \ref{lem:sheaf} (i) implies that $\Phi (\mathcal{O}_x)$ is a shift of a sheaf.
Hence,  \cite[Lemma 4.3]{Br99}  implies that $\mathcal{P}|_{p_1^{-1}(S\backslash Z)}$ 
is a shift of a sheaf, flat over $S\backslash Z$ by $p_1$.
Consequently, we see  $\dim (\Phi(\mc{O}_x))=0$.
Furthermore, assume $\Phi\in \ker \iota_Z$. Then, we can say  $\dim (\Phi(\mc{O}_x))=0$ for all point $x\in S$, and thus
 it follows from  \cite[Lemma 2.2]{Ue16} that  
$$
\Phi \cong \phi _*\circ((-)\otimes \mathcal{L})
$$  
for $\mathcal{L}\in \Pic _Z(S)$, $\phi\in \Aut_Z(S)$.
Therefore, we obtain the result. 
\end{proof}

\begin{conj}\label{conj:generation}
Let $S$ be a smooth projective surface. Then
$$\Auteq^\dagger_{\Kequiv} D_Z(S)=\Span{\Br_Z(S),(\Aut (S)/\Aut_Z (S))}\rtimes (\Pic (S)/\Pic _Z(S))\times \Z [1].$$
Consequently,  
$$
\Auteq_{\Kequiv} D(S)=\Span{\Br_Z(S),\Pic (S)} \rtimes \Aut (S)\times \Z [1]. 
$$
\end{conj}


\begin{thm}\label{thm:generation}
Let $S$ be a smooth projective surface.
Then Conjecture \ref{conj:generation} holds true, if 
$Z$ is a disjoint union of configurations of $(-2)$-curves of type $A$.
\end{thm}

\begin{proof}
Note that $\Phi \in \Auteq^\dagger_{\Kequiv} D_Z(S)$ preserves the class $\ch (\mc{O}_x)$ for $x\in Z$.
Then for any point $x\in Z$, Lemma \ref{lem:sheaf} (ii) assures that  there is  an autoequivalence $\Psi\in \Br_Z(S)$, 
an integer  $i$ and a point $y\in Z$ 
satisfying $\Psi\circ\Phi (\mc{O}_x)\cong \mc{O}_y[i]$.  
Now the result follows from \cite[Corollary 5.23]{Hu06}.
\end{proof}

Recall that if $\Auteq D(S)$ is $K$-equivalent type, then 
$\Auteq D(S)=\Auteq_{\Kequiv}D(S)$. Therefore, Theorem \ref{thm:generation} implies Theorem \ref{thm:generation0}.


\noindent
Department of Mathematics
and Information Sciences,
Tokyo Metropolitan University,
1-1 Minamiohsawa,
Hachioji-shi,
Tokyo,
192-0397,
Japan 

{\em e-mail address}\ : \  hokuto@tmu.ac.jp
\end{document}